\documentclass{amsart}
\usepackage{amssymb,amsmath,amscd,amsthm}
\usepackage{amsfonts,epsfig,latexsym,graphicx,amssymb}

\date{\today}

\newcommand{\Z}{{\mathbb Z}}
\newcommand{\R}{{\mathbb R}}

\newcommand{\Q}{{\mathbb Q}}

\newtheorem{theorem}{Theorem}[section]
\newtheorem{remark}[theorem]{Remark}
\newtheorem{lemma}[theorem]{Lemma}
\newtheorem{prop}[theorem]{Proposition}

\newtheorem{defi}[theorem]{Definition}
\newtheorem{conj}[theorem]{Conjecture}
\newtheorem{question}[theorem]{Question}
\def\be{\begin{equation}}
\def\ee{\end{equation}}

\begin{document}

\title[The Dimension of the Spectrum of Sturmian Hamiltonians]{Almost Sure Frequency Independence of the Dimension of the Spectrum of Sturmian Hamiltonians}

\author[D.\ Damanik]{David Damanik}

\address{Department of Mathematics, Rice University, Houston, TX~77005, USA}

\email{damanik@rice.edu}

\thanks{D.\ D.\ was supported in part by NSF grant DMS--1067988.}

\author[A.\ Gorodetski]{Anton Gorodetski}

\address{Department of Mathematics, University of California, Irvine, CA~92697, USA}

\email{asgor@math.uci.edu}

\thanks{A.\ G.\ was supported in part by NSF grants DMS--1301515 and 
IIS-1018433.}

\begin{abstract}
We consider the spectrum of discrete Schr\"odinger operators with Sturmian potentials and show that for sufficiently large coupling, its Hausdorff dimension and its upper box counting dimension are the same for Lebesgue almost every value of the frequency.
\end{abstract}

\maketitle

\section{Introduction}\label{s.intro}

In this paper we study discrete Schr\"odinger operators
$$
[H_{\lambda,\alpha,\omega}\psi](n) = \psi(n+1) + \psi(n-1) + \lambda \chi_{[1 - \alpha,1)} (n \alpha + \omega \!\!\!\!\mod 1) \psi(n)
$$
in $\ell^2(\Z)$, where $\lambda > 0$ is the coupling constant, $\alpha \in (0,1) \setminus \Q$ is the frequency, and $\omega \in [0,1)$ is the phase. These operators are popular models of one-dimensional quasicrystals and have been studied since the 1980's; see, for example, \cite{B92, BIST89, BIT91, D07, DEG} and references therein.

By the minimality of irrational rotations of the circle and strong operator convergence it follows that the spectrum of $H_{\lambda,\alpha,\omega}$ is independent of $\omega$ and may therefore be denoted by $\Sigma_{\lambda,\alpha}$. The spectrum does, however, depend on $\lambda$ and $\alpha$. It is known from \cite{BIST89} that $\Sigma_{\lambda,\alpha}$ is a Cantor set of zero Lebesgue measure. One is therefore naturally interested in the fractal dimension of this set.

The following conjecture has been around for quite a long time:\footnote{Jean Bellissard informs us that he has expected this statement to hold since the 1980's.}

\begin{conj}\label{c.main}
For every $\lambda > 0$, the Hausdorff dimension of $\Sigma_{\lambda,\alpha}$ is Lebesgue almost everywhere constant in $\alpha$.
\end{conj}

The intuition behind this conjecture is that the Hausdorff dimension of $\Sigma_{\lambda,\alpha}$, as a function of $\alpha$ for $\lambda$ fixed, should be invariant under the Gauss map. Indeed, the spectrum displays scaling properties that are very well understood and which are determined by the continued fraction expansion of $\alpha$. These scaling properties suggest that the Hausdorff dimension should only depend on the tail of the continued fraction expansion of $\alpha$, and hence the invariance under the Gauss map. The statement about Lebesgue almost everywhere constancy then follows from the existence of an ergodic measure for the Gauss map which is mutually absolutely continuous relative to Lebesgue measure.

The same reasoning applies to the upper box counting dimension of the spectrum, and hence one can reasonably expect this quantity to be almost everywhere constant in the frequency as well. Note that it is known that these two dimensions of the spectrum do not coincide in general \cite{LPW07}.

Alas, not even any partial results are known. In this paper we will prove the almost everywhere constancy of $\dim_H \Sigma_{\lambda,\alpha}$ and $\dim_B^+ \Sigma_{\lambda,\alpha}$ in $\alpha$ for $\lambda$ sufficiently large.

\begin{theorem}\label{t.main}
For every $\lambda \ge 24$, both the Hausdorff dimension and the upper box counting dimension of $\Sigma_{\lambda,\alpha}$ are Lebesgue almost everywhere constant in $\alpha$.
\end{theorem}

Our proof makes crucial use of recent work of Liu, Qu and Wen \cite{LQW14} who identified, for $\lambda \ge 24$, $\dim_H \Sigma_{\lambda,\alpha}$ and $\dim_B^+ \Sigma_{\lambda,\alpha}$ with the so-called lower and upper pre-dimensions, which are defined via Moran-type formulas. These identities enable us to indeed prove the invariance of $\dim_H \Sigma_{\lambda,\alpha}$ and $\dim_B^+ \Sigma_{\lambda,\alpha}$ under the Gauss map since the pre-dimensions studied by Liu, Qu and Wen are quite amenable to this question.

The overall strategy, namely to show the invariance of the dimension of the spectrum (or some other spectral quantity of interest) under the Gauss map, is applicable whenever the potentials are dynamically defined, with the underlying dynamics given by an irrational rotation of the circle. While this approach is natural and known to experts in the field, this is the first time it has actually been implemented. We hope that this will lead to subsequent successful implementations of this strategy for other quantities and/or sampling functions; compare Questions~\ref{q.amo} and \ref{q.general} in Section~\ref{s.remarks}.

The main idea behind the proof of Theorem~\ref{t.main}, along with a natural strengthening of this result that comes naturally out of this proof, Theorem~\ref{t.main2}, is described in Section~\ref{s.gaussmap}. In Section~\ref{s.lqw} we recall the necessary definitions and results from \cite{LQW14} that are needed in our discussion, including the principle of bounded covariance. Readers familiar with the work \cite{LQW14} may skip ahead to Section~\ref{s.invariance}. We show in Section~\ref{s.invariance} that bounded covariance implies the invariance of the dimensions in question under the Gauss map. Theorems~\ref{t.main} and \ref{t.main2} are then proved in Section~\ref{s.proof}. Finally, Section~\ref{s.remarks} contains some questions and open problems arising from and related to Theorems~\ref{t.main} and \ref{t.main2}.

\section{Invariance Under the Gauss Map and Almost Everywhere Constancy with Respect to any Ergodic Measure}\label{s.gaussmap}

In \cite{LQW14}, Liu, Qu and Wen work with the so-called lower and upper pre-dimensions $s_*(\lambda,\alpha)$ and $s^*(\lambda,\alpha)$ of $\Sigma_{\lambda,\alpha}$. The definition of these pre-dimensions is somewhat involved and we will recall it in \eqref{e.sdef} in Section~\ref{s.lqw} below. The main result from \cite{LQW14} we will rely on is the following.

\begin{theorem}[Liu-Qu-Wen 2014]\label{t.lqw}
For $\lambda \ge 24$ and $\alpha \in (0,1) \setminus \Q$, we have $\dim_H \Sigma_{\lambda,\alpha} = s_*(\lambda,\alpha)$ and $\dim_B^+(\Sigma_{\lambda,\alpha}) = s^*(\lambda,\alpha)$.
\end{theorem}

This result suggests that we should study the invariance of $s_*(\lambda,\alpha)$ and $s^*(\lambda,\alpha)$ under the Gauss map
$$
G : \alpha \mapsto \left\{ \frac{1}{\alpha} \right\},
$$
where $\{ x \}$ denotes the fractional part of $x$, that is, $\{ x \} = x - \lfloor x \rfloor$.

Each irrational $\alpha \in (0,1)$ has a continued fraction expansion
$$
\alpha = \cfrac{1}{a_1 + \cfrac{1}{a_2 + \cfrac{1}{a_3 + \cdots}}} =: [a_1,a_2,a_3,\ldots] ,
$$
with uniquely determined $a_k \in \Z_+ = \{ 1, 2, 3, \ldots \}$.

Note that the Gauss map, restricted to $\alpha \in (0,1) \setminus \Q$, results in a shift of the sequence of continued fraction coefficients, that is, it sends $[a_1,a_2,a_3,\ldots]$ to $[a_2,a_3,a_4,\ldots]$.

We will show the following invariance result:

\begin{theorem}\label{t.invariance}
For $\lambda \ge 24$ and $\alpha \in (0,1) \setminus \Q$, we have $s_*(\lambda,\alpha) = s_*(\lambda,G(\alpha))$ and $s^*(\lambda,\alpha) = s^*(\lambda,G(\alpha))$.
\end{theorem}

Theorems~\ref{t.lqw} and \ref{t.invariance} imply the invariance of $\dim_H \Sigma_{\lambda,\alpha}$ and $\dim_B^+(\Sigma_{\lambda,\alpha})$ under the Gauss map, for each fixed $\lambda \ge 24$. Since these dimensions are, for $\lambda$ fixed, measurable functions of $\alpha$ (see Theorem~\ref{t.mm} below), we obtain the following strengthening of Theorem~\ref{t.main}.

\begin{theorem}\label{t.main2}
Suppose $\mu$ is an ergodic Borel probability measure for the Gauss map $G$. Then, for every $\lambda \ge 24$, both the Hausdorff dimension and the upper box counting dimension of $\Sigma_{\lambda,\alpha}$ are $\mu$-almost everywhere constant in $\alpha$.
\end{theorem}

Theorem~\ref{t.main} follows from Theorem~\ref{t.main2} since the Gauss map has an ergodic measure which is mutually absolutely continuous relative to Lebesgue measure. Of course $G$ has many other ergodic measures and Theorem~\ref{t.main2} gives information about the typical behavior with respect to any of them as well.
As a simple illustration let us consider an atomic ergodic invariant measure of $G$ (i.e., a measure supported on a periodic orbit of $G$). In this case Theorem~\ref{t.main2} implies that for any two frequencies with eventually periodic continued fraction expansion with common periodic part, the dimensions of the spectrum must be the same. This fact was known before \cite{G14, Me14} for all $\lambda > 0$, not only for $\lambda \ge 24$.

As a different example, consider all frequencies that have only $1$'s and $2$'s in their continued fraction expansion. Then Theorem~\ref{t.main2} implies that for almost all of them (with respect to a Bernoulli measure on $\{ 1,2 \}^{\Z_+}$), the Hausdorff dimension of the spectrum is the same (and the same statement holds for the upper box counting dimension).

\section{Generating Bands, Symbolic Coding, and Pre-Dimensions}\label{s.lqw}

In this section we summarize the parts of \cite{LQW14} we will need in our proof of Theorem~\ref{t.invariance}. Naturally, the presentation will be succinct and we refer the reader to \cite{LQW14} and its predecessors \cite{DEGT, FLW11, LPW07, LW04, LW05, R97} for motivation and more extensive discussion of the definitions below.

Fix a coupling constant $\lambda > 0$ and a frequency $\alpha \in (0,1) \setminus \Q$. Consider the continued fraction expansion of the frequency, $\alpha = [a_1,a_2,a_3,\ldots]$. The $k$-th partial quotient of $\alpha$ is given by $\frac{p_k}{q_k}$, where
\begin{align*}
p_{-1} & = 1, \; p_0 = 0, \; p_{k+1} = a_{k+1} p_k + p_{k-1}, \; k \ge 0, \\
q_{-1} & = 0, \; q_0 = 1, \; q_{k+1} = a_{k+1} q_k + q_{k-1}, \; k \ge 0.
\end{align*}

For $E \in \R$, define the matrices
$$
M_{-1}(E) = \begin{pmatrix} 1 & - \lambda \\ 0 & 1 \end{pmatrix}, \; M_0(E) = \begin{pmatrix} E & - 1 \\ 1 & 0 \end{pmatrix}, \; M_{k+1}(E) = M_{k-1}(E) M_k(E)^{a_{k+1}}, \; k \ge 0.
$$
Furthermore, define for $k \ge 0$, $p \ge -1$,
$$
t_{(k,p)}(E) = \mathrm{Tr} \left( M_{k-1}(E) M_k(E)^p \right), \; \sigma_{(k,p)} = \{ E \in \R : | t_{(k,p)}(E) | \le 2 \}.
$$
Then, $t_{(k,p)}$ is a polynomial of degree $p q_k + q_{k-1}$ which is strictly monotone on each connected component $B_{(k,p)}$ of $\sigma_{(k,p)}$ and, moreover, $t_{(k,p)}(B_{(k,p)}) = [-2,2]$. Such a connected component $B_{(k,p)}$ of $\sigma_{(k,p)}$ is usually called a band.

Let now $\lambda > 4$. Then we can define for $k \ge 0$ three families of bands:
\begin{itemize}

\item A band of $\sigma_{(k,1)}$ that is contained in a band of $\sigma_{(k,0)}$ is of type (k,I).

\item A band of $\sigma_{(k+1,0)}$ that is contained in a band of $\sigma_{(k,-1)}$ is of type (k,II).

\item A band of $\sigma_{(k+1,0)}$ that is contained in a band of $\sigma_{(k,0)}$ is of type (k,III).

\end{itemize}
Denote by $\mathcal{G}_k$ the collection of these bands; they are called the spectral generating bands of order $k$. We have
\begin{equation}\label{e.specdesc}
\Sigma_{\lambda,\alpha} = \bigcap_{k \ge 0} \bigcup_{B \in \mathcal{G}_k} B.
\end{equation}

For $k \ge 1$, define $s_k$ to be the unique number in $[0,1]$ for which
\begin{equation}\label{e.skdef}
\sum_{B \in \mathcal{G}_k} |B|^{s_k} = 1.
\end{equation}
The lower and upper pre-dimensions of $\Sigma_{\lambda,\alpha}$ are defined by
\begin{equation}\label{e.sdef}
s_*(\lambda,\alpha) = \liminf_{k \to \infty} s_k, \quad s^*(\lambda,\alpha) = \limsup_{k \to \infty} s_k.
\end{equation}

The matrix
\begin{equation}\label{e.matrix}
\begin{pmatrix} 0 & 1 & 0 \\ a_{k+1} + 1 & 0 & a_{k+1} \\ a_{k+1} & 0 & a_{k+1} - 1 \end{pmatrix}
\end{equation}
governs that combinatorial relation between the three types of bands as one moves from level $k$ to $k+1$, that is, any band in $\mathcal{G}_k$ that is of type (k,I) contains $0$ bands of $\mathcal{G}_{k+1}$ that are of type (k+1,I), $1$ band of $\mathcal{G}_{k+1}$ that is of type (k+1,II), and $0$ bands of $\mathcal{G}_{k+1}$ that are of type (k+1,III), and so on. This suggests a symbolic coding. Consider the possible transitions,
$$
\mathcal{E} = \{ (\mathrm{I,II}), (\mathrm{II,I}), (\mathrm{II,III}), (\mathrm{III,I}), (\mathrm{III,III}) \},
$$
and define for $e \in \mathcal{E}$ and $n \in \Z_+$,
$$
\tau_e(n) = \begin{cases} 1 & e = (\mathrm{I,II}) \\ n+1 & e = (\mathrm{II,I}) \\ n & e = (\mathrm{II,III}) \\ n & e = (\mathrm{III,I}) \\ n-1 & e = (\mathrm{III,III}) \end{cases}.
$$
Then define
\begin{align*}
\mathcal{E}_n & = \{ (e,\tau_e(n),\ell) : e \in \mathcal{E}, 1 \le \ell \le \tau_e(n) \}, \\
\mathcal{E}_n^* & = \{ (e,\tau_e(n),\ell) \in \mathcal{E}_n : e \not= (\mathrm{II,I}), (\mathrm{II,III}) \}.
\end{align*}
For any $n, n' \in \Z_+$ and $(e, \tau_e(n), \ell) \in \mathcal{E}_n$, $(e', \tau_{e'}(n'), \ell') \in \mathcal{E}_{n'}$, we say that $(e, \tau_e(n), \ell) (e', \tau_{e'}(n'), \ell')$ is admissible, denoted by $(e, \tau_e(n), \ell) \to (e', \tau_{e'}(n'), \ell')$, if the second component of $e$ and the first component of $e'$ coincide.

Set $\Omega^{(\alpha)}_1 = \mathcal{E}_{a_1}^*$ and, for $k \ge 2$,
$$
\Omega^{(\alpha)}_k = \left\{ \omega \in \mathcal{E}_{a_1}^* \times \prod_{j=2}^k  \mathcal{E}_{a_j} : \omega = \omega_1 \omega_2 \ldots \omega_k \text{ s.t. } \omega_j \to \omega_{j+1} \text{ for every } 1 \le j \le k-1 \right\},
$$
and $\Omega^{(\alpha)}_* = \bigcup_{k \in \Z_+} \Omega^{(\alpha)}_k$.

Given any $w \in \Omega^{(\alpha)}_k$, we define the associated band $B_w$ inductively as follows. Let $B_\mathrm{I} = [\lambda - 2,\lambda + 2]$ be the unique band in $\mathcal{G}_0$ of type (0,I) and let $B_\mathrm{III} = [-2,2]$ be the unique band in $\mathcal{G}_0$ of type (0,III). Suppose $w \in \Omega^{(\alpha)}_1$. If $w = ((\mathrm{I,II}), 1, 1)$, then define $B_w$ to be the unique band of type (1,II) that is contained in $B_\mathrm{I}$. If $w = ((\mathrm{III,I}), \tau_{(\mathrm{III,I})}(a_1), \ell)$, then define $B_w$ to be the $\ell$-th band (counted from the left) of type (1,I) that is contained in $B_\mathrm{III}$. Finally, if $w = ((\mathrm{III,III}), \tau_{(\mathrm{III,III})}(a_1), \ell)$, then define $B_w$ to be the $\ell$-th band (counted from the left) of type (1,III) that is contained in $B_\mathrm{III}$. Next, if $B_{w'}$ has been defined for all $w' \in \Omega^{(\alpha)}_{k-1}$, given $w \in \Omega^{(\alpha)}_k$, write $w = w' (e,\tau_e(a_k),\ell)$ with $w' \in \Omega^{(\alpha)}_{k-1}$. Writing $e = (\mathrm{T,T'})$, define $B_w$ to be the $\ell$-th band of type (k,$\mathrm{T'}$) inside $B_{w'}$. With these notations, \eqref{e.specdesc} can be rewritten as
\begin{equation}\label{e.specdesc2}
\Sigma_{\lambda,\alpha} = \bigcap_ {k \ge 0} \bigcup_{w \in \Omega^{(\alpha)}_k} B_w.
\end{equation}

One of the crucial technical tools in \cite{LQW14} is the principle of bounded covariation; see Theorem~3.3 in that paper:

\begin{theorem}[Liu-Qu-Wen 2014]\label{t.boundedcovariation}
There are absolute constants $C_1,C_2$ such that for $\lambda \ge 24$ and $\alpha \in (0,1) \setminus \Q$, we have
$$
\eta^{-1} \frac{|B_{wu}|}{|B_w|} \le \frac{|B_{\tilde wu}|}{|B_{\tilde w}|} \le \eta \frac{|B_{wu}|}{|B_w|},
$$
where $\eta = C_1 \exp( 2 C_2 \lambda)$, whenever $w, wu, \tilde w, \tilde w u \in \Omega^{(\alpha)}_*$.
\end{theorem}

\section{Bounded Covariation Implies Invariance of Pre-Dimensions}\label{s.invariance}

We will need the following simple statement.

\begin{prop}\label{p.simple}
Suppose $\mathcal{B}_k=\{b_k(n)\}_{n=1, \ldots, N_k}$, $k \in \Z_+$, is a collection of positive numbers such that $\sum_{n=1}^{N_k}b_k(n)< 1$ for $k \ge k_0$ and $\max_{n=1, \ldots N_k}b_k(n)\to 0$ as $k\to \infty$. Let $s_k$ be the unique solution of the equation $\sum_{n=1}^{N_k}(b_k(n))^{s_k}=1$. Set
\be\label{e.sstar1}
s_*=\liminf_{k\to \infty} s_k, \ \ \ s^*=\limsup_{k\to \infty}s_k.
\ee
Suppose $\mathcal{D}_k=\{d_k(n)\}_{n=1, \ldots, N_k}$, $k\in \Z_+$, is a collection of positive numbers such that $\sum_{n=1}^{N_k}d_k(n)< 1$ for $k \ge k_1$ and $\max_{n=1, \ldots N_k}d_k(n)\to 0$ as $k\to \infty$. Let $\delta_k$ be the unique solution of the equation $\sum_{n=1}^{N_k}(d_k(n))^{\delta_k}=1$. Set
\be\label{e.deltastar}
\delta_*=\liminf_{k\to \infty} \delta_k, \ \ \ \delta^*=\limsup_{k\to \infty}\delta_k.
\ee
If there exists $C\ge 1$ such that for any $k\in \Z_+$ and $n\in \{1, \ldots, N_k\}$, we have
$$
C^{-1}\le \frac{b_k(n)}{d_k(n)}\le C,
$$
then $s_*=\delta_*$ and $s^*=\delta^*$.
\end{prop}

\begin{proof}
Notice that from the assumptions it follows that $0\le s_k<1$ and $0\le \delta_k<1$ for all $k \ge \max \{ k_0, k_1 \}$. Fix any $\varepsilon > 0$. Then, for all sufficiently large $k \in \Z_+$, we have
\begin{align*}
\sum_{n=1}^{N_k}(d_k(n))^{s_k+\varepsilon} & \le \sum_{n=1}^{N_k}(Cb_k(n))^{s_k+\varepsilon} \\
& \le C^{s_k+\varepsilon}\left[\sum_{n=1}^{N_k}(b_k(n))^{s_k}\right]\left(\max_{n=1, \ldots, N_k} b_k(n)\right)^\varepsilon \\
& = C^{s_k+\varepsilon}\left(\max_{n=1, \ldots, N_k} b_k(n)\right)^\varepsilon \\
& < 1,
\end{align*}
and hence $\delta_k<s_k+\varepsilon$. Similarly, $s_k<\delta_k+\varepsilon$, so $|s_k-\delta_k|<\varepsilon$ for sufficiently large values of $k$. Therefore,
$$
s_* = \liminf_{k \to \infty} s_k = \liminf_{k \to \infty} \delta_k = \delta_*
$$
and
$$
s^* = \limsup_{k \to \infty} s_k = \limsup_{k \to \infty} \delta_k = \delta^*,
$$
concluding the proof.
\end{proof}

\begin{defi}
For a set $K\subset \mathbb{R}$ and a point $x \in K$, define
$$
\dim_\mathrm{H,\,loc}(K, x) = \lim_{\varepsilon \to 0} \dim_\mathrm{H} \left( K \cap (x-\varepsilon, x+\varepsilon) \right).
$$
Similarly we define
$$
\dim^+_\mathrm{B,\,loc} (K, x) = \lim_{\varepsilon \to 0} \dim^+_\mathrm{B} \left( K \cap (x-\varepsilon, x+\varepsilon) \right).
$$
\end{defi}

The following statement is of independent interest.

\begin{prop}\label{p.localdim}
For any $x \in \Sigma_{\lambda, \alpha}$, we have
$$
\dim_\mathrm{H,\,loc} (\Sigma_{\lambda, \alpha}, x) = \dim_\mathrm{H} \Sigma_{\lambda, \alpha}, \ \ \ \dim^+_\mathrm{B,\,loc} (\Sigma_{\lambda, \alpha}, x) = \dim^+_\mathrm{B} \Sigma_{\lambda, \alpha}.
$$
\end{prop}

\begin{proof}[Proof of Proposition \ref{p.localdim}]
We will start the proof with two lemmas.
\begin{lemma}\label{l.first}
Suppose that $B$ and $B'$ are two bands of the same type $(\tilde k, T)$ for some $\tilde k\in \mathbb{N}$ and $T\in \{I, II, III\}$. Then
$$
\dim_\mathrm{H} (B \cap \Sigma_{\lambda, \alpha}) = \dim_\mathrm{H} (B' \cap \Sigma_{\lambda, \alpha}) \ \ \text{and} \ \ \dim_\mathrm{B}^+ (B \cap \Sigma_{\lambda, \alpha}) = \dim_\mathrm{B}^+ (B' \cap \Sigma_{\lambda, \alpha}).
$$
\end{lemma}

\begin{proof}[Proof of Lemma~\ref{l.first}]
Let $\{b_k(n)\}_{n=1, \ldots, N_k}$ be the set of lengths of all bands of level $k$ for a given $k\ge \tilde k$ contained in the band $B$ and ordered with respect to the combinatorics described in Section \ref{s.lqw}. Let $\{d_k(n)\}_{n=1, \ldots, N_k}$ be the set of lengths of all bands of level $k$ contained in the band $B'$ and ordered in the same way. Theorem~\ref{t.boundedcovariation} implies that the ratios $\frac{b_k(n)}{d_k(n)}$ are uniformly bounded. Hence Proposition~\ref{p.simple} implies that $s_*=\delta_*$ and $s^*=\delta^*$, where $s^*, s_*,  \delta^*, \delta_*$ are defined by (\ref{e.sstar1}) and (\ref{e.deltastar}).

Lemma~\ref{l.first} now follows from the fact that \cite[Theorem~1.1]{LQW14} holds for the spectrum $\Sigma_{\lambda, \alpha}$ within the band $B$ (as well as within the band $B'$), that is, $\dim_\mathrm{H} (\Sigma_{\lambda, \alpha} \cap B) = s_*$ and $\dim^+_\mathrm{B} (\Sigma_{\lambda, \alpha} \cap B) = s^*$. Indeed, the proofs of \cite[Proposition~1]{LW05} (where the inequalities $\dim_\mathrm{H} \Sigma_{\lambda, \alpha} \le s_*$ and  $\dim^+_{B} \Sigma_{\lambda, \alpha} \ge s^*$ are shown), of the inequality $\dim^+_\mathrm{B} \Sigma_{\lambda, \alpha} \le s^*$ given in \cite[Section~4.1]{LQW14}, and of the inequality $\dim_\mathrm{H} \Sigma_{\lambda, \alpha} \ge s_*$ given in \cite[Section~4.2]{LQW14} hold verbatim for the sequence of bands generated in $B$ and the spectrum within $B$.
\end{proof}

\begin{lemma}\label{l.second}
For any $x\in \Sigma_{\lambda, \alpha}$ and any $\varepsilon>0$ there exists $k'\in \mathbb{N}$ such that for any $k\ge k'$ and any $T\in \{I, II, III\}$ there exists a band of type $(k, T)$ that belongs to the interval $(x-\varepsilon, x+\varepsilon)$.
\end{lemma}

\begin{proof}[Proof of Lemma~\ref{l.second}]
The description of the spectrum $\Sigma_{\lambda, \alpha}$ given by (\ref{e.specdesc}) implies that there exists a band $B$ of some level $\bar k$ inside of $(x-\varepsilon, x+\varepsilon)$. Notice that the cube of the matrix (\ref{e.matrix}) does not have any zero entries. This implies that for any $k\ge \bar k+3$ and any $T\in \{I, II, III\}$ the band $B$ contains a band of type $(k, T)$.
\end{proof}
In order to prove Proposition~\ref{p.localdim} it is enough to show that for any $x \in \Sigma_{\lambda, \alpha}$ and any $\varepsilon > 0$, we have
\begin{equation*}
\begin{split}
 \ & \dim_\mathrm{H} (\Sigma_{\lambda, \alpha} \cap (x - \varepsilon, x + \varepsilon)) \ge \dim_\mathrm{H} \Sigma_{\lambda, \alpha}, \ \ \ \text{and} \ \ \ \\
 \ & \dim^+_\mathrm{B} (\Sigma_{\lambda, \alpha} \cap (x - \varepsilon, x + \varepsilon)) \ge \dim^+_\mathrm{B} \Sigma_{\lambda, \alpha}.
\end{split}
\end{equation*}
Let $k' \in \mathbb{N}$ be given by Lemma~\ref{l.second}. Then the interval $(x-\varepsilon, x+\varepsilon)$ contains some bands of type $(k', I)$,  $(k', II)$, and  $(k', III)$. Denote them by $B_I$, $B_{II}$, and $B_{III}$.
Let $\{B_j\}_{j=1, \ldots, N_k}$ be a collection of all bands of level $k'$. Then
\begin{equation*}
\begin{split}
&\dim_\mathrm{H} \Sigma_{\lambda, \alpha} = \max_{j=1, \ldots, N_k} \dim_\mathrm{H} (\Sigma_{\lambda, \alpha} \cap B_j), \\
&\dim_\mathrm{B}^+ \Sigma_{\lambda, \alpha} = \max_{j=1, \ldots, N_k} \dim_\mathrm{B}^+ (\Sigma_{\lambda, \alpha} \cap B_j).
\end{split}
\end{equation*}
Due to Lemma \ref{l.first}, for all $j$, the value $\dim_\mathrm{H} (\Sigma_{\lambda, \alpha} \cap B_j)$ must be equal to $\dim_\mathrm{H} (\Sigma_{\lambda, \alpha} \cap B_I)$, or $\dim_\mathrm{H} (\Sigma_{\lambda, \alpha} \cap B_{II})$, or $\dim_\mathrm{H} (\Sigma_{\lambda, \alpha} \cap B_{III})$. Therefore,
\begin{align*}
\dim_\mathrm{H} \Sigma_{\lambda, \alpha} & = \max \{\dim_\mathrm{H} (\Sigma_{\lambda, \alpha} \cap B_I), \dim_\mathrm{H}  (\Sigma_{\lambda, \alpha} \cap B_{II}), \dim_\mathrm{H} (\Sigma_{\lambda, \alpha} \cap B_{III}) \} \\
& \le \dim_\mathrm{H} (\Sigma_{\lambda, \alpha} \cap (x - \varepsilon, x + \varepsilon))
\end{align*}
and
\begin{align*}
\dim_\mathrm{B}^+ \Sigma_{\lambda, \alpha} & = \max \{ \dim_\mathrm{B}^+ (\Sigma_{\lambda, \alpha} \cap B_I),  \dim_\mathrm{B}^+ (\Sigma_{\lambda, \alpha} \cap B_{II}), \dim_\mathrm{B}^+ (\Sigma_{\lambda, \alpha} \cap B_{III}) \} \\
& \le \dim_\mathrm{B}^+ (\Sigma_{\lambda, \alpha} \cap (x - \varepsilon, x + \varepsilon)).
\end{align*}
This completes the proof of Proposition~\ref{p.localdim}.
\end{proof}

Now we are ready to prove Theorem~\ref{t.invariance}.

\begin{proof}[Proof of Theorem~\ref{t.invariance}]
Denote $\beta = G(\alpha)$. If $\alpha = [a_1, a_2, a_3, \ldots]$, then $\beta = [a_2, a_3, \ldots]$. Consider $B_\mathrm{I} = [\lambda - 2,\lambda + 2]$ -- the  band of type $(0,I)$ and $B_\mathrm{III} = [-2,2]$ -- the band of type $(0, III)$. Set $w = ((III, I), a_1, 1)$, then $B_w$ is one of the $a_1$ bands of type $(1, I)$ contained in $B_{III}$ generating $\Sigma_{\lambda, \alpha}$. Notice that if $wu \in \Omega^{(\alpha)}_*$, then $B_u \subseteq B_I$. Notice also that the set of bands contained in $B_{I}$ generating $\Sigma_{\lambda, \beta}$ is exactly the set of bands of the form $B_{u}$, where $wu \in \Omega^{(\alpha)}_*$. By Theorem~\ref{t.boundedcovariation}, for any sequence $u \in \Omega^{(\beta)}_*$ such that $wu \in \Omega^{(\alpha)}_*$, we have
$$
\eta^{-1} \frac{|B_{u}|}{|B_I|} \le \frac{|B_{wu}|}{|B_{w}|} \le \eta \frac{|B_{u}|}{|B_I|}.
$$
If there are $N_k$ possible sequences $u \in \Omega^{(\beta)}_*$ of length $k$ such that $wu \in \Omega^{(\alpha)}_*$, denote the length of the bands $B_{wu}$ and $B_u$ by $b_k(n)$ and $d_k(n)$, respectively, $n=1, \ldots, N_k$. From the proof of Lemma~\ref{l.first} one can see that the values of $s_*, s^*$ and $\delta_*, \delta^*$ generated by these collections of numbers as in \eqref{e.sstar1}, \eqref{e.deltastar} will be the same as $s_*(\lambda, \alpha), s^*(\lambda, \alpha)$, and $s_*(\lambda, \beta), s^*(\lambda, \beta)$. Now an application of Proposition~\ref{p.simple} completes the proof of Theorem~\ref{t.invariance}.
\end{proof}

\section{Proof of the Main Theorems}\label{s.proof}

In this section we prove Theorems~\ref{t.main} and \ref{t.main2}. The main input will be provided by Theorem~\ref{t.invariance}. In addition, we will need to address the measurability of the quantities in question.

The following continuity result is due to Bellissard, Iochum and Testard \cite{BIT91}. Denote by $\mathcal{K}(\R)$ the space of compact subsets of the real line, equipped with the Hausdorff metric.

\begin{theorem}[Bellissard, Iochum, Testard 1991]\label{t.bit}
For every $\lambda > 0$, the map
$$
(0,1) \setminus \Q \ni \alpha \mapsto \Sigma_{\lambda,\alpha} \in \mathcal{K}(\R)
$$
is continuous.
\end{theorem}

This result is not stated in this exact way in \cite{BIT91}, but Theorem~\ref{t.bit} may be derived from \cite[Theorem~1]{BIT91} in a straightforward way.

Mapping the spectrum to its Hausdorff dimension or upper box counting dimension is a measurable operation, as shown by Mattila and Mauldin \cite[Theorem~2.1.(b) and Lemma~3.1]{MM97}:

\begin{theorem}[Mattila, Mauldin 1997]\label{t.mm}
The maps
$$
\mathcal{K}(\R) \ni K \mapsto \dim_\mathrm{H} K \in [0,1]
$$
and
$$
\mathcal{K}(\R) \ni K \mapsto \dim_\mathrm{B}^+ K \in [0,1]
$$
are Baire class $2$ functions. In particular, they are both Borel functions.
\end{theorem}

We are now able to prove the remaining theorems.

\begin{proof}[Proof of Theorem~\ref{t.main2}]
Suppose $\lambda \ge 24$. By Theorem~\ref{t.invariance}, each of the maps
\begin{equation}\label{e.2maps}
(0,1) \setminus \Q \ni \alpha \mapsto \dim_\mathrm{H} \Sigma_{\lambda,\alpha} \in [0,1] , \quad (0,1) \setminus \Q \ni \alpha \mapsto \dim_\mathrm{B}^+ \Sigma_{\lambda,\alpha} \in [0,1]
\end{equation}
is invariant with respect to the Gauss map $G$. Moreover, by Theorems~\ref{t.bit} and \ref{t.mm}, each of the maps is Borel. Thus, for any $G$-ergodic Borel probability measure $\mu$, each of these two maps is $\mu$-almost everywhere constant.
\end{proof}

\begin{proof}[Proof of Theorem~\ref{t.main}]
The Gauss measure
$$
\mu(B) = \frac{1}{\log 2} \int_B \frac{1}{1+x} \, dx
$$
is a Borel probability measure on $(0,1) \setminus \Q$, which is ergodic with respect to $G$; see, for example, \cite{EW11}. Thus, by Theorem~\ref{t.main2}, the two maps in \eqref{e.2maps} are $\mu$-almost everywhere constant. Since $\mu$ is clearly mutually absolutely continuous relative to Lebesgue measure, Theorem~\ref{t.main} follows.
\end{proof}

\section{Some Open Problems}\label{s.remarks}

In this final section we list some questions and open problems related to Theorem~\ref{t.main}.

It would be of great interest to do away with the large coupling assumption, and we expect that it should be possible:

\begin{conj}\label{c.1}
The assumption $\lambda \ge 24$ can be dropped in Theorem~\ref{t.main}.
\end{conj}

Since our proofs rely so heavily on \cite{LQW14}, which in turn uses the largeness assumption in an essential way (e.g., it is clearly impossible to treat all $\lambda > 0$ in this way as the hierarchical band structure requires $\lambda > 4$), a proof of this conjecture is currently well out of reach.

Moreover, it is natural to ask

\begin{question}
Can one determine the {\rm (}$\lambda$-dependent{\rm )} almost everywhere values of the dimensions in Theorem~\ref{t.main}?
\end{question}

This appears to be quite difficult as this would require a detailed quantitative understanding of the length of the spectral generating bands. It is likely significantly easier to identify the almost everywhere behavior in the large coupling limit. It was shown in \cite{LQW14} that for every $\alpha \in (0,1) \setminus \Q$, there are $f_*(\alpha)$ and $f^*(\alpha)$ such that
$$
\lim_{\lambda \to \infty} s_*(\lambda,\alpha) \log \lambda = - \log f_*(\alpha), \; \lim_{\lambda \to \infty} s^*(\lambda,\alpha) \log \lambda = - \log f^*(\alpha).
$$
By this result and Theorem~\ref{t.main}, there are $f_*$ and $f^*$ such that
$$
\lim_{\lambda \to \infty} s_*(\lambda,\alpha) \log \lambda = - \log f_*, \; \lim_{\lambda \to \infty} s^*(\lambda,\alpha) \log \lambda = - \log f^*
$$
for Lebesgue almost every $\alpha \in (0,1) \setminus \Q$.

\begin{question}
Can these numbers $f_*,f^*$ be determined explicitly?
\end{question}

There are explicit formulas for the numbers $f_*(\alpha)$ and $f^*(\alpha)$ (see \cite{LQW14}, see also \cite{DEGT, Q14} for explicit values in the case of metallic means, i.e., frequencies of constant type) and hence it may be quite possible to determine their almost sure behavior.

We also believe that it is reasonable to state the following

\begin{conj}
For every $\lambda > 0$, the spectra $\Sigma_{\lambda, \alpha}$ and $\Sigma_{\lambda, G(\alpha)}$ are diffeomorphic, that is, there exist neighborhoods $U \supset \Sigma_{\lambda, \alpha}$ and $V \supset \Sigma_{\lambda, G(\alpha)}$ and a $C^1$ diffeomorphism  $f : U \to V$ such that $f(\Sigma_{\lambda, \alpha}) = \Sigma_{\lambda, G(\alpha)}$.
\end{conj}

Notice that the conjecture is equivalent to the statement that for any two irrational numbers $\alpha$ and $\beta$ such that $G^k(\alpha) = G^m(\beta)$ for some $k, m\in \mathbb{Z}^+$, the spectra $\Sigma_{\lambda, \alpha}$ and $\Sigma_{\lambda, \beta}$ are diffeomorphic for every $\lambda$. In particular, that would imply that any property of the spectrum that is invariant under diffeomorphisms must be a ``tail property'' of the continued fraction expansion of $\alpha$.

Clearly, a positive answer to this conjecture will provide a strengthening of Theorems~\ref{t.main}, \ref{t.invariance}, \ref{t.main2} and also prove Conjecture~\ref{c.1}.

\medskip

Next, we can also look beyond properties of $\Sigma_{\lambda,\alpha}$ and consider other important quantities associated with the operator family $\{ H_{\lambda,\alpha,\omega} \}_{\omega \in [0,1)}$.

The density of states measure $dN_{\lambda,\alpha}$ is a probability measure supported by $\Sigma_{\lambda,\alpha}$, which is defined by
\begin{equation}\label{e.idsdef}
\int g(E) \, dN_{\lambda,\alpha}(E) = \int_\omega \langle \delta_0 , g(H_{\lambda,\alpha,\omega}) \delta_0 \rangle \, d\omega.
\end{equation}
In fact, $dN_{\lambda,\alpha}$ is the equilibrium measure associated with $\Sigma_{\lambda,\alpha}$ in the sense of logarithmic potential theory (this follows from the vanishing of the Lyapunov exponent on the spectrum \cite{BIST89} and a general result of Simon \cite{S07}).

\begin{question}
Given $\lambda > 0$, is the map
\begin{equation}\label{e.dimdsm}
(0,1) \setminus \Q \ni \alpha \mapsto \dim_\mathrm{H} dN_{\lambda,\alpha} \in [0,1]
\end{equation}
measurable and invariant under the Gauss map? In particular, is it true that for every $\lambda > 0$, $\dim_\mathrm{H} dN_{\lambda,\alpha}$ takes the same value for Lebesgue almost every $\alpha$?
\end{question}

Here, $\dim_\mathrm{H} dN_{\lambda,\alpha}$ denotes the upper Hausdorff dimension of $dN_{\lambda,\alpha}$, that is, the infimum over all Hausdorff dimensions of sets supporting $dN_{\lambda,\alpha}$.

Transport exponents capture the rate of spreading in the time-dependent Schr\"odinger equation. For $p > 0$, consider the $p$-th moment of the position operator,
$$
\langle |X|_{\delta_0}^p \rangle (t) = \sum_{n \in \Z} |n|^p | \langle e^{-itH_{\lambda,\alpha,\omega}} \delta_0 , \delta_n \rangle |^2
$$
We average in time as follows. If $f(t)$ is a function of $t > 0$ and $T > 0$ is given, we denote the time-averaged function at $T$ by $\langle f \rangle (T)$:
$$
\langle f \rangle (T) = \frac{2}{T} \int_0^{\infty} e^{-2t/T} f(t) \, dt.
$$
Then, the corresponding upper and lower transport exponents $\tilde \beta^+_{\delta_0}(p)$ and $\tilde \beta^-_{\delta_0}(p)$ are given, respectively, by
$$
\tilde \beta^+_{\delta_0}(p) = \limsup_{T \to \infty} \frac{\log \langle \langle |X|_{\delta_0}^p \rangle \rangle (T) }{p \, \log T},
$$
$$
\tilde \beta^-_{\delta_0}(p) = \liminf_{T \to \infty} \frac{\log \langle \langle |X|_{\delta_0}^p \rangle \rangle (T) }{p \, \log T}.
$$
The transport exponents $\tilde \beta^\pm_{\delta_0}(p)$ belong to $[0,1]$ and are non-decreasing in $p$ (see, e.g., \cite{DT10}), and hence the following limits exist:
\begin{align*}
\tilde \alpha_l^\pm & = \lim_{p \to 0} \tilde \beta^\pm_{\delta_0}(p), \\
\tilde \alpha_u^\pm & = \lim_{p \to \infty} \tilde \beta^\pm_{\delta_0}(p).
\end{align*}
All these transport exponents depend on $\lambda,\alpha,\omega$. It is quite possible that the dependence on $\omega$ is trivial for some of them. For example, in the case of $\alpha = \frac{\sqrt{5}-1}{2}$ and $\tilde \alpha_u^\pm$, this was in fact shown in \cite{DGY14}. On the other hand, the dependence on $\lambda$ is known to be non-trivial \cite{DKL, DT07}. Regarding the dependence on $\alpha$, keeping with the theme of this paper, the following question arises:

\begin{question}
Given $\lambda > 0$, which of these transport exponents depend on $\alpha$ in a measurable and $G$-invariant way? Specifically, which of these transport exponents are, for given $\lambda > 0$, Lebesgue almost everywhere constant in $\alpha$? {\rm (}Here one would either have to show independence of $\omega$ or possibly average in $\omega$.{\rm )}
\end{question}

\begin{remark}
One can also consider the optimal global H\"older exponent $\gamma_{\lambda,\alpha}$ of the density of states measure $dN_{\lambda,\alpha}$. It follows from Munger's recent work \cite{M14} that for $\lambda \ge 24$, the optimal H\"older exponent is the same for Lebesgue almost every $\alpha$, and in fact this almost sure value is zero. He also relies on the line of papers \cite{FLW11, LPW07, LQW14, LW04, R97}.
\end{remark}

Going even beyond the operator family $\{ H_{\lambda,\alpha,\omega} \}_{\omega \in [0,1)}$, one can ask similar questions about other operator families. For example, the heavily studied almost Mathieu operator is given by
$$
[H^\mathrm{AMO}_{\lambda,\alpha,\omega}\psi](n) = \psi(n+1) + \psi(n-1) + 2 \lambda \cos (2 \pi(n \alpha + \omega)) \psi(n),
$$
where $\lambda > 0$, $\alpha \in (0,1) \setminus \Q$, and $\omega \in [0,1)$. It follows again easily that the spectrum of $H^\mathrm{AMO}_{\lambda,\alpha,\omega}$ is independent of $\omega$ and may therefore be denoted by $\Sigma^\mathrm{AMO}_{\lambda,\alpha}$. It is known that
$$
\mathrm{Leb} (\Sigma^\mathrm{AMO}_{\lambda,\alpha}) = 4 | 1- \lambda |.
$$
In particular, it is of interest to determine the Hausdorff dimension of $\Sigma^\mathrm{AMO}_{\lambda,\alpha}$ for $\lambda = 1$. There are some known upper bounds for $\dim_\mathrm{H} \Sigma^\mathrm{AMO}_{1,\alpha}$ that hold for set of frequencies $\alpha$ of Lebesgue measure zero \cite{L94, S10}. However, for Lebesgue almost every $\alpha$ there is as yet no information on $\dim_\mathrm{H} \Sigma^\mathrm{AMO}_{1,\alpha}$. The same strategy as the one pursued in this paper, namely showing the invariance under the Gauss map, may be applicable here and yield the almost everywhere constancy of $\dim_\mathrm{H} \Sigma^\mathrm{AMO}_{1,\alpha}$.

\begin{question}\label{q.amo}
Is $\dim_\mathrm{H} \Sigma^\mathrm{AMO}_{1,\alpha}$ invariant under $G$? In particular, is it true that $\dim_\mathrm{H} \Sigma^\mathrm{AMO}_{1,\alpha}$ takes the same value for Lebesgue almost every $\alpha$?
\end{question}

Furthermore, one can ask the same questions about the density of states measure and transport exponents as above for the almost Mathieu operator. 

More generally, we have the following open-ended question, which was already implicitly addressed in the introduction. A one-frequency quasi-periodic Schr\"odinger operator is an operator in $\ell^2(\Z)$ of the form
$$
[H_{f,\alpha,\omega}\psi](n) = \psi(n+1) + \psi(n-1) + f(n \alpha + \omega) \psi(n),
$$
where $f : \R / \Z \to \R$ is bounded and measurable, $\alpha \in (0,1) \setminus \Q$, and $\omega \in [0,1)$.

\begin{question}\label{q.general}
For which one-frequency quasi-periodic Schr\"odinger operators are there spectral quantities of interest that depend measurably on the frequency and are invariant under the Gauss map?
\end{question}

\section*{Acknowledgment}

We would like to express our gratitude to Jean Bellissard from whom we learned about Conjecture~\ref{c.main} and the intuition behind it.

\end{document}